\documentclass[11pt, a4paper]{article}
\usepackage[utf8]{inputenc}

\usepackage{amsmath,amssymb,amsthm}
\usepackage{mathtools}
\usepackage{thmtools}
\usepackage{thm-restate}

\usepackage{fullpage}
\usepackage{xcolor}
\usepackage[shortlabels]{enumitem}
\usepackage{hyperref, cleveref}
\hypersetup{hidelinks}
\usepackage[skip=7pt plus2pt]{parskip}

\usepackage{mleftright}
\mleftright

\newtheorem{theorem}{Theorem}[section]

\newtheorem{lemma}[theorem]{Lemma}

\newtheorem{cor}[theorem]{Corollary}

\theoremstyle{definition}

\theoremstyle{remark}
\newtheorem{claim}{Claim}[theorem]
\newtheorem*{claim*}{Claim}

\newenvironment{claimproof}[1][Proof of Claim]
  {%
    \proof[#1]%
  }
  {%
    \endproof%
  }

\usepackage{biblatex}
\DeclareFieldFormat*{pages}{#1}
\renewbibmacro{in:}{}
\DeclareFieldFormat*{title}{`#1'}

\title{Layered subgraphs of the hypercube}
\author{Natalie Behague \and Imre Leader \and Natasha Morrison\thanks{Research supported by NSERC Discovery Grant RGPIN-2021-02511 and NSERC Early Career Supplement DGECR-2021-00047.} \and Kada Williams}
\newcommand{\Addresses}{{% additional braces for segregating \footnotesize
  \bigskip
  \footnotesize

  Natalie Behague, Mathematics institute, University of Warwick, Coventry, UK. \\ \nopagebreak
  \textit{Email}: \href{mailto:natalie.behague@warwick.ac.uk}{natalie.behague@warwick.ac.uk}

  \medskip
    
    Imre~Leader, Department of Pure Mathematics and Mathematical Statistics, Centre for Mathematical Sciences, Wilberforce Road, Cambridge CB3 0WB, UK. \\ \nopagebreak
    \textit{Email}:
    \href{mailto: i.leader@dpmms.cam.ac.uk}{ i.leader@dpmms.cam.ac.uk}
 
    \medskip

    Natasha Morrison, Department of Mathematics and Statistics, University of Victoria, Canada. \\ \nopagebreak
    \textit{Email}: \href{mailto:nmorrison@uvic.ca}{nmorrison@uvic.ca}

    \medskip
    
    Kada Williams, Department of Pure Mathematics and Mathematical Statistics, Centre for Mathematical Sciences, Wilberforce Road, Cambridge CB3 0WB, UK. \\ \nopagebreak
    \textit{Email}: \href{mailto:kkw25@cam.ac.uk}{kkw25@cam.ac.uk}
}}

\date{\today}

\begin{document}

\maketitle

\begin{abstract}
    A subgraph of the $n$-dimensional hypercube is called `layered' if it is a subgraph of a layer of some hypercube. In this paper we show that there exist subgraphs of the cube of arbitrarily large girth that are not layered. This answers a question of Axenovich, Martin and Winter. Perhaps surprisingly, these subgraphs  may even be taken to be induced. 
\end{abstract}

\section{Introduction}

%\tash{Add something about turan?}

Recall that the $n$-dimensional \emph{hypercube} $Q_n$ has vertex set $\{0, 1\}^n$ and two vertices $x, y \in Q_n$ are joined by an edge if and only if they differ on exactly one coordinate. In this paper, we are interested in subgraphs of the hypercube that cannot be found as subgraphs of a `layer' of any hypercube. 

To be more precise, the $k$-th \emph{layer} of $Q_n$ is the subgraph induced by all vertices in $Q_n$ containing exactly $k$ or $k+1$ 1s in their coordinate representation. We say that a graph $G$ is \emph{cubical} if there exists some $n$ such that $G$ is a subgraph of the hypercube $Q_n$, and we say that $G$ is \emph{layered} if there exists some $n$ and $k$ such that $G$ is a subgraph of the $k$th edge-layer of $Q_n$.

Whether or not a given cubical graph is layered is an important question that relates to Tur\'an density problems. Indeed, if a cubical
connected graph $G$ is not layered then it is easy to see that it has
positive Tur\'an edge-density, meaning that the greatest number of
edges in a $G$-free subgraph of $Q_n$ is not $o(e(Q_n))$ -- we
just take alternate layers of $Q_n$. It is a major open question
to determine which cubical graphs have positive Tur\'an edge-density:
we mention for example the recent proof by Grebennikov and Marciano~\cite{grebennikov2024c10} that the $10$-cycle 
has positive Tur\'an density, which was the last remaining cycle
for which the answer was not known. 
See Axenovich, Martin and Winter~\cite{axenovich2023graphs} for
background and many related results. 

Now, there are many graphs that are cubical but not layered -- a simple
example is the 4-cycle. However, the known examples all tend to fail to
be layered for `local' reasons: they contain 4-cycles, or else some
other dense cube-like local structure. Indeed, one can find
graphs of girth 6 that are cubical but not layered, but getting beyond this seems difficult. 
Axenovich, Martin and Winter~\cite{axenovich2023graphs} 
exhibited an elegant graph of girth 8 that is cubical but not layered. They asked if there exist graphs of arbitrarily large girth that are cubical but not layered -- in other words, cubical graphs that fail to be layered for a `long-distance' reason. 
Our aim in this paper is to answer this question in the affirmative.

\begin{restatable}{theorem}{main}
 \label{thm:main}
    For every $k$ there exists a graph of girth at least $k$ that is cubical but not layered.
\end{restatable}
In fact, we also obtain a strengthening of this result that tells us this is even possible with induced subgraphs. This is perhaps rather surprising. 

\begin{restatable}{theorem}{induced} \label{thm:induced}
    For every $k$ there exists a graph of girth at least $k$ that is an induced subgraph of some hypercube but is not layered.
\end{restatable}

A key notion in the proof of Theorem~\ref{thm:main} is the
consideration of metric
properties of embeddings. That is, we consider how the graph distance (in the subgraph) relates to the actual metric distance in the hypercube. At the heart of our proof is an iterative procedure (see Lemma~\ref{lem:recursion}) that takes a starting cubical graph with particular properties and creates a cubical graph that can only be embedded in a layer if certain pairs of vertices are `far' apart. A suitable modification of the graph obtained at the end of this process can then be shown to be cubical but not layered. The proof of Theorem~\ref{thm:induced} is more involved, as we need to preserve certain additional properties during the proof (as well
as starting with a different base graph).

In Section~\ref{sec:mainthm} we prove Theorem~\ref{thm:main}. The further work required for the proof of Theorem~\ref{thm:induced} is given in Section~\ref{sec:indthm}. We use standard graph-theoretic notation -- see for example~\cite{bollobas2013modern}. For an edge $xy$ of the hypercube, we say its \emph{direction} is the coordinate in
which $x$ and $y$ differ.

\section{Proof of Theorem~\ref{thm:main}}\label{sec:mainthm}

We first describe a relationship between particular embeddings of a graph $G$ into a hypercube and certain edge-labellings of $G$. This will be useful for us later in showing whether $G$ is cubical or layered. 

We start with some notation. An \emph{embedding} $f$ of a graph $G$ into a graph $H$ (which, for our purposes, will be taken to be the hypercube $Q_n$ or a layer of $Q_n$) is a graph isomorphism from $G$ to a subgraph of $H$. 

As noted by Havel and Morávek~\cite{havel1972}, one may characterise cubical graphs in terms of edge-labellings; in the following definitions we slightly abuse notation, using `cubical' and 'layered' as properties of certain edge-labellings. For any graph $G$, 
an edge-labelling $\chi:E(G) \rightarrow \mathbb{N}$ of $G$ is called \emph{cubical} if it has the following properties.
	\begin{enumerate}[(i)]
		\item For any set of edges $e_1,e_2,\ldots, e_k$ forming a cycle, each possible label occurs an even number of times in $\chi(e_1), \chi(e_2), \ldots, \chi(e_k)$.
		\item For any (non-empty) set of edges $e_1,e_2,\ldots, e_k$ forming a path, some label occurs an odd number of times in $\chi(e_1), \chi(e_2), \ldots, \chi(e_k)$.
	\end{enumerate}

\begin{lemma}[Havel and Morávek~\cite{havel1972}]  \label{lem:labelling_for_cube}
    Let $G$ be a graph. An embedding $f$ of $G$ into a hypercube $Q_n$ induces a cubical edge-labelling  $\chi:E(G) \rightarrow \mathbb{N}$, where $\chi(xy)$ is the direction of the edge $f(x)f(y)$ in $Q_n$. Conversely, if $G$ admits a cubical edge-labelling  $\chi$ then there exists a embedding $f$ of $G$ into a hypercube such that $\chi(xy)$ is the direction of the edge $f(x)f(y)$.
\end{lemma}

Axenovitch, Martin and Winter~\cite{axenovich2023graphs} made the
very useful observation that one may characterise layered graphs
in a similar way. For a graph $G$, we call an edge-labelling $\chi:E(G) \rightarrow \mathbb{N}$ of $G$  \emph{layered} if it is cubical and has the additional property that for any for any two edges of the same colour $c$, any path between them with no edges of colour $c$ contains an even number of edges.

\begin{lemma}[Axenovitch, Martin and Winter~\cite{axenovich2023graphs}] \label{lem:labelling_for_layer}
  Let $G$ be a graph. An embedding $f$ of $G$ into a layer of a hypercube $Q_n$ induces a layered edge-labelling  $\chi:E(G) \rightarrow \mathbb{N}$,  where $\chi(xy)$ is the direction of the edge $f(x)f(y)$ in $Q_n$. Conversely, if $G$ admits a layered edge-labelling  $\chi$ then there exists a embedding $f$ of $G$ into a layer of a hypercube such that $\chi(xy)$ is the direction of the edge $f(x)f(y)$.
\end{lemma}

We now introduce an important notion that will be used in building our graphs of large girth that are cubical but not layered. In the
following, we stress that, for two vertices $x$ and $y$ of a cube,
we always write $d(x,y)$ to denote the distance from $x$ to $y$ in the cube itself -- never the distance inside a particular subgraph.

Let $G$ be a graph and let $P$ be a set of unordered pairs of vertices in $G$. (We do not insist that these pairs are disjoint.)  For a given positive integer $t$, we say that $(G,P)$ is \emph{$t$-distance-separating} if the following hold:
    \begin{enumerate}[(i)]
        \item There exists some $n$ and some embedding $f$ of $G$ into $Q_n$ such that $d(f(x),f(y)) = t$ for all pairs $\{x,y\} \in P$; and
        \item For any embedding $g$ of $G$ into a layer of a hypercube, there is some pair $\{x,y\} \in P$ with $d(g(x),g(y)) = t+2$.
    \end{enumerate}

The following is the key lemma for the proof of Theorem~\ref{thm:main}. The reader may worry that it concerns only trees, and of course all trees are layered -- but in fact in our construction we use trees until the final step, and it is only then
that we move away from trees.

\begin{lemma} \label{lem:recursion}
    Let $T$ be a tree and let $P$ be a set of pairs of vertices in $T$. Suppose that $(T,P)$ is $t$-distance-separating, for
    some value of $t$. 
    Then there exists a tree $T'$ and a set $P'$ of pairs of vertices in $T'$ such that $(T',P')$ is $(t+2)$-distance-separating.
\end{lemma}
\begin{proof}
    Let $s = t+4$.      
    Let $T'$ be the tree obtained from $T$ as follows. For each vertex $x$ in $T$, we 
    add $s$ new leaves $x_1,\ldots, x_s$ adjacent to $x$. Thus $|T'| = (s+1)|T|$. Let $P_{xy}$ be the set of all pairs $\{ x_i,y_j \}$ where $1 \le i,j \le s$, and let $P'$ be the union of $P_{xy}$ taken over all pairs $\{x,y\}$ in $P$. 
 
    We must check that $(T',P')$ is $(t+2)$-distance-separating.

    \begin{claim} \label{claim:recursion_cubical}
        There exists an $n$ and an embedding $f$ of $T'$ into $Q_n$ such that  for all pairs $\{x,y\} \in P'$ we have ${d(f(x),f(y)) = t+2}$.
    \end{claim}
    \begin{claimproof}  

            We first define a cubical edge-labelling of $T'$ in order to apply Lemma~\ref{lem:labelling_for_cube} to obtain an embedding $f$ of $T'$ into some hypercube. We then show that $f$ satisfies the required distance properties.
            
            As $T$ is $t$-distance-separating, there is some $n$ and some isomorphism $g$ from $T$ to an induced subgraph of $Q_n$ such that $d(g(x),g(y)) = t$ for all pairs $\{x,y\} \in P$. By Lemma~\ref{lem:labelling_for_cube}, the embedding $g$ corresponds to a cubical edge-labelling of $T$. Let the edges of $T'\cap T$ inherit these labels. 

            Give each edge in $E(T')\setminus E(T)$ a unique new label distinct from any other label. It is straightforward to verify that the resulting edge-labelling is cubical: $T'$ does not contain any cycles so the first condition is satisfied automatically. Any path containing an edge of $T' \setminus T$ has a label used exactly once, namely the label on the new edge. Any path not containing a new edge is entirely in $T$, for which the labelling was cubical. 
            
            In total we introduce $s|T|$ new labels. Let $f$ be the corresponding embedding of $T'$ into the hypercube $Q_{n+s|T|}$ given by Lemma~\ref{lem:labelling_for_cube}. %It is easy to see that $f'(T')$ is an induced subgraph of $Q_{n+2s|P|}$, as $f(T)$ was an induced subgraph of $Q_{n}$ and all new edges use distinct directions.
            
            Let $\{u,v\} \in P'$. We must show that $d(f(u),f(v)) = t+2$. We have that $\{u,v\}$ is in $P_{xy}$ for some $\{x,y\} \in P$ with $u$ adjacent to $x$ and $v$ adjacent to $y$. We see that $f(u)$ differs from $f(x)$ in one coordinate that is the label used on $ux$ and similarly,  $f(v)$ differs from $f(y)$ in one coordinate that is the label used on $vy$. By construction, these labels are distinct from each other and from the coordinates that $g(x)$ and $g(y)$ differ in (which are the same as the coordinates that  $g(x)$ and $g(y)$ differ in). Therefore, 
            \[d(f(u),f(v)) = d(f(x),f(y)) + 2 = d(g(x),g(y)) + 2 = t+2.\]
    \end{claimproof}

    \begin{claim} \label{claim:recursion_layer}
        For any embedding $g$ of $T'$ into a layer of a hypercube, there is some pair $\{x,y\} \in P'$ with $d(g(x),g(y)) = t+4$.
    \end{claim}
    \begin{claimproof}
        Let $g$ be an embedding of $T'$ into a layer $L$ of a hypercube. As $T$ is $t$-distance-separating, we know that there is some $\{x,y\} \in P$ with $d(g(x),g(y)) = t+2$. We will show that there is some $\{u,v\} \in P_{xy}$ for this $\{x,y\}$ with $d(g(u),g(v)) = t+4$. 

        Let $I$ be the set of coordinates where $g(x)$ and $g(y)$ differ, so $|I| = d(g(u),g(v)) = t+2$. Note that the directions of edges adjacent to $g(x)$ must all be distinct and thus we can find at least $s - |I| = 2$ vertices adjacent to $x$ that are not in $T$ and do not use a direction in $I$ -- let $u$ be one of these. Similarly, we can find $2$ vertices adjacent to $y$ that are not in $T$ and do not use a direction in $I$ -- call these $v_1$ and $v_2$. Take $v \in \{v_1,v_2\}$ such that $yv$ has a label different from the label on $xu$. We see that 
        \[d(g(u),g(v)) = d(g(x),g(y)) + 2 = t+4.\]
    \end{claimproof}
    
    Claims~\ref{claim:recursion_cubical}~and~\ref{claim:recursion_layer} together give that $(T',P')$ is $(t+2)$-distance-separating, completing the proof of 
    Lemma~\ref{lem:recursion}. 
    \end{proof}

We now turn to the `base case'.

\begin{lemma} \label{lem:base_case}
    Let $P_6$ be the path on $6$ edges with vertices labelled $v_0,v_1,\ldots,v_6$, and let $P$ be the pairs of vertices $\{ \{v_0,v_4\}, \{v_1,v_5\}, \{v_2,v_6\}\}$. Then $(P_6,P)$ is $2$-distance-separating.
\end{lemma}
\begin{proof}
	Let $f$ be the following embedding of $P_6$ into $Q_3$:
	\[
	v_0 \mapsto \emptyset, \qquad
	v_1 \mapsto \{1\}, \qquad
	v_2 \mapsto \{1,2\}, \qquad
	v_3 \mapsto \{2\}, \]
	\[
	v_4 \mapsto \{2,3\}, \qquad
	v_5 \mapsto \{1,2,3\}, \qquad
	v_6 \mapsto \{1,3\}.
	\]
	It is easy to check that $d(f(x),f(y)) = 2$ for all $\{x,y\} \in P$, and so the first condition is satisfied.
	
	To prove the second condition, we will use the following claim.
 
 \begin{claim} \label{claim:labelP_6}
     Let $\chi:E(P_6) \rightarrow \mathbb{N}$ be any layered edge-labelling of $P_6$. Then for some $0 \le i \le 2$, there are four distinct labels on the edges $v_iv_{i+1}, v_{i+1}v_{i+2}, v_{i+2}v_{i+3}, v_{i+3}v_{i+4}$.
 \end{claim}
 \begin{claimproof}[Proof of Claim]
 This is just a routine case-check; we write it out merely for
 completeness.
 
 Recall that for an edge-labelling to be layered the following three conditions must hold:
 \begin{enumerate}[(i)]
     	\item For any set of edges $e_1,e_2,\ldots, e_k$ forming a cycle, each possible label occurs an even number of times in $\chi(e_1), \chi(e_2), \ldots, \chi(e_k)$, \label{condition:cycle}
		\item For any (non-empty) set of edges $e_1,e_2,\ldots, e_k$ forming a path, some label occurs an odd number of times in $\chi(e_1), \chi(e_2), \ldots, \chi(e_k)$, \label{condition:path}
        \item for any for any two edges of the same colour $c$, any path between them that has no edges of colour $c$ has an even length. \label{condition:odd_distance}
 \end{enumerate}

     Without loss of generality, let edge $v_0v_1$ have label $1$. By condition~\ref{condition:path} on the path $v_0v_1v_2$, edge $v_1v_2$ has a different label to $1$, say $2$. 
     
     Edge $v_2v_3$ cannot have label $2$ by condition~\ref{condition:path} on the path $v_1v_2v_3$, and cannot have label $1$ by condition~\ref{condition:odd_distance} with edge $v_0v_1$. So it must have a new label, say $3$.

     Edge $v_3v_4$ cannot have label $3$ by condition~\ref{condition:path} on the path $v_2v_3v_4$, and cannot have label $2$ by condition~\ref{condition:odd_distance} with edge $v_1v_2$. If it had a label other than $1,2,3$ then edges $v_0v_{1}, v_{1}v_{2}, v_{2}v_{3}, v_{3}v_{4}$ would all have distinct labels and we would be done. So we may suppose that $v_2v_3$ has label $1$.

    Edge $v_4v_5$ cannot have label $1$ by condition~\ref{condition:path} on the path $v_3v_4v_5$, and cannot have label $3$ by condition~\ref{condition:odd_distance} with edge $v_2v_3$. If it had a label other than $1,2,3$ then edges $v_{1}v_{2}, v_{2}v_{3}, v_{3}v_{4},v_4v_5$ would all have distinct labels and we would be done. So we may suppose that $v_2v_3$ has label $2$. See Figure~\ref{fig:p6} for an illustration of the labelling forced up to this point.

    \begin{figure}[h]
    \centering
    \includegraphics[scale=1.1]{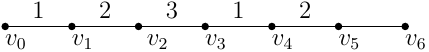}
    \caption{The edge-labels forced on $P_6$ during the proof of Lemma~\ref{lem:base_case}.}\label{fig:p6}
    \end{figure}

    Edge $v_5v_6$ cannot have label $2$ by condition~\ref{condition:path} on the path $v_4v_5v_6$, and cannot have label $1$ by condition~\ref{condition:odd_distance} with edge $v_3v_4$. It also cannot have label $3$ by condition~\ref{condition:path} on the path $v_0v_1v_2v_3v_4v_5v_6$.  Therefore it must have a label distinct from $1,2,3$. Then the edges $v_{2}v_{3}, v_{3}v_{4},v_4v_5, v_5v_6$ all have distinct labels and we are done.
 \end{claimproof}

 Let $f$ be an embedding of $P_6$ into a layer $L$ of a hypercube. Let  $\chi: E(G) \rightarrow \mathbb{N}$ be the edge-labelling where $\chi(uv)$ is the direction of the edge $f(u)f(v)$, which is layered by Lemma~\ref{lem:labelling_for_layer}.  By Claim~\ref{claim:labelP_6}, there is some pair of points $\{x,y\} \in P$ where there are four distinct labels on the edges of the path joining them. Therefore, for this $x,y$ we have $d(f(x),f(y)) = 4$.
\end{proof}

Lemmas~\ref{lem:recursion} and~\ref{lem:base_case} together immediately give the following corollary.
\begin{cor}\label{cor:main}
    For all even $t \ge 2$, there exists a tree $T$ and a  set $P$ of pairs of vertices in $T$ such that $(T,P)$ is $t$-distance-separating.
\end{cor}

Armed with Corollary~\ref{cor:main}, we can now prove Theorem~\ref{thm:main} which we restate below for convenience.
\main*
\addtocounter{theorem}{1}
\begin{proof}
  Fix  $k \ge 2$ to be even. By Corollary~\ref{cor:main}, there exists a tree $T$ and a set $P$ of pairs of vertices in $T$ that is $(k-2)$-distance-separating.  

    Define a graph $G$ as follows. For each pair of vertices $\{x,y\}$ in $P$, add $k+1$ paths of $k$ edges between $x$ and $y$ that are vertex disjoint other than at the endpoints. Let $S_{xy}$ be this set of $k+1$ paths, which we will refer to as a \emph{spindle}. See Figure~\ref{fig:spindle} for an illustration of a spindle.

        \begin{figure}[h]
    \centering
    \includegraphics[width=0.6\textwidth]{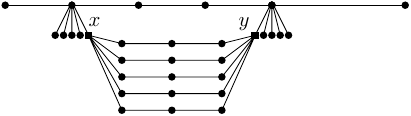}
    \caption{A spindle (with $k=4$) between vertices $x$ and $y$.}\label{fig:spindle}
    \end{figure}

    As $T$ is a tree, any cycle in $G$ must use the entirety of one of these added paths and so clearly $G$ has girth at least $k$.

    \begin{claim}\label{claim:main_cubical}
        $G$ is cubical.
    \end{claim}
    \begin{claimproof}
        We define a cubical edge-labelling of $G$ in order to apply Lemma~\ref{lem:labelling_for_cube}. As $T$ is $(k-2)$-distance-separating, there exists an $n$ and an embedding $f$ of $T$ into $Q_n$ such that we have ${d(f(x),f(y)) = k-2}$ for all pairs $\{x,y\} \in P$. By Lemma~\ref{lem:labelling_for_cube}, this embedding corresponds to a cubical edge-labelling of $T$. The edges of $G$ contained in the underlying copy of $T$ (that is, those edges that are not in an added spindle) inherit these labels. 
            
        Fix $\{x,y\} \in P$. Let $I_{xy}$ be the set of coordinates that $f(x)$ and $f(y)$ differ on. Note that $|I_{xy}| = d(f(x),f(y)) = k-2$.           
        For each of the paths in $S_{xy}$, we introduce a new edge label, distinct from those used elsewhere in $G'$, and use it to label the first and last edges of the path. Label the central $k-2$ edges of the path by the elements of $I_{xy}$ in any order.
        
        Repeat this for each pair $\{x,y\} \in P$, so that every edge of $G$ has a label. We must verify that the resulting edge-labelling is cubical. As $T$ is a tree, any cycle in $G$ must use the entirety of one of the added paths. Thus a cycle in $G$ is the union of two paths from $x$ to $y$ for some $\{x,y\} \in P$. A path from $x$ to $y$ uses each label in $I_{xy}$ an odd number of times and all other labels an even number of times, so the union of two such paths uses every label an even number of times. 

        Suppose for a contradiction there is a non-empty path $Q$ in $G$ where every label is used an even number of times on $Q$. We know that $Q$ cannot be contained entirely within $T$, as the edge-labelling on $T$ is cubical. Thus there must be some $\{x,y\} \in P$ such that $Q$ intersects a path $Q' \in S_{xy}$. As $Q'$ starts and ends with a label that is not used elsewhere in $G$, we must have that $Q$ contains both the start and end edges of $Q'$. 

        If $Q$ does not contain the entirety of $Q'$, then $Q$ must consist of a path from $x$ to $y$ together with some strict subset of the edges of $Q'$. In this case, there is some direction in $I_{xy}$ that is used an odd number of times.

        Thus we may assume that for all $\{x,y\} \in P$ and for all $Q' \in S_{xy}$, the path $Q$ contains either all edges of $Q'$ or none of them. Consider the walk $\overline{Q}$ on $T$ obtained by taking $Q$ and replacing every traversal of a path in $S_{xy}$ by the path in $T$ from $x$ to $y$. We have that $\overline{Q}$ uses each label an even number of times. The label of an edge in $T$ corresponds to the direction on the  image of the edge under the embedding $f$, and so this means that the walk $\overline{Q}$ starts and ends at the same vertex. However, then $Q$ must also start and end at the same vertex, which contradicts that $Q$ is a path. 
    \end{claimproof}

    \begin{claim} \label{claim:main_layered}
        $G$ is not layered.
    \end{claim}
    \begin{claimproof}
        Suppose for a contradiction that $g$ is an embedding of $G$ into an edge-layer of a hypercube. As $T$ is $(k-2)$-distance-separating, we know that there is some $\{x,y\} \in P$ with $d(g(x),g(y)) = k$. Let $I$ be the set of coordinates that $g(x)$ and $g(y)$ differ on, so that $|I| = k$.        

        There are $k+1$ paths from $g(x)$ to $g(y)$ of length $k$ that are vertex disjoint other than at the endpoints. However, a path of length $k$ from $g(x)$ to $g(y)$ must use exactly the $k$ directions in $I$. In particular, these $k+1$ paths must each start with one of $k$ distinct directions from $I$. Hence we have a contradiction, and $G$ is not layered.
    \end{claimproof}

This concludes the proof of Theorem~\ref{thm:main}.  
\end{proof}

\section{Proof of Theorem~\ref{thm:induced}}\label{sec:indthm}
The family of graphs used to prove Theorem~\ref{thm:main} are not induced subgraphs of the hypercube. Thus it is natural to ask whether there exist induced subgraphs of the hypercube of arbitrarily large girth that are not layered. 
It turns out that there are such graphs, which we prove by extending the framework of the previous section. 

\begin{lemma} \label{lem:base_case_induced}
    Let $T$ be the tree on $9$ vertices that is the path $v_0,v_1,\ldots,v_7$ together with a vertex $v_8$ adjacent to $v_2$ and let \[P = \{ \{v_0,v_4\}, \{v_0,v_6\}, \{v_3,v_7\},  \{v_5,v_8\}\}.\] 
    Then $(T,P)$ is $2$-distance-separating. Moreover, there exists some embedding $f$ of $T$ into $Q_4$ such that    
    \begin{enumerate}[(i)]
            \item $d(f(x),f(y)) = 2$ for all pairs $\{x,y\} \in P$, 
            \item for any pair of points $\{x,y\}$ in $P$, each direction is used on at most two edges on the path from $f(x)$ to $f(y)$ in $f(T)$,    and      
            \item $f(T)$ is an induced subgraph of $Q_{4}$.
    \end{enumerate}
\end{lemma}
See Figure~\ref{fig:induced_base_case} for a diagram of the tree $T$.

\begin{proof}
 	Let $f$ be the following embedding of $T$ into $Q_4$:
	\[
	v_0 \mapsto \emptyset, \qquad
	v_1 \mapsto \{1\}, \qquad
	v_2 \mapsto \{1,2\}, \qquad
	v_3 \mapsto \{1,2,3\}, \qquad
        v_4 \mapsto \{2,3\},\]
	\[	
	v_5 \mapsto \{2,3,4\}, \qquad
	v_6 \mapsto \{3,4\}, \qquad
        v_7 \mapsto \{1,3,4\}, \qquad
        v_8 \mapsto \{1,2,4\}.
	\]
	It is easy to check that $f(T)$ is an induced subgraph of $Q_{4}$ as two vertices $u,v$ are adjacent in $T$ if and only if $f(u)$ and $f(v)$ are adjacent in $Q_4$. Moreover,  for all $\{x,y\} \in P$, we have $d(f(x),f(y)) = 2$ and each direction is used on at most two edges on the path from $f(x)$ to $f(y)$ in $f(T)$. 
	
To complete the proof we will use the following claim. 
 \begin{claim} \label{claim:label_T}
     Let $\chi:E(T) \rightarrow \mathbb{N}$ be any layered edge-labelling of $T$. Then for some $\{x,y\} \in P$, there exist four labels each of which is used exactly once on the edges of the path between $x$ and $y$.
 \end{claim}
 \begin{claimproof}[Proof of Claim]
 This is a routine check. 
 Recall that for an edge-labelling to be layered the following three conditions must hold:
 \begin{enumerate}[(i)]
     	\item For any set of edges $e_1,e_2,\ldots, e_k$ forming a cycle, each possible label occurs an even number of times in $\chi(e_1), \chi(e_2), \ldots, \chi(e_k)$, \label{condition:cycle_2}
		\item For any (non-empty) set of edges $e_1,e_2,\ldots, e_k$ forming a path, some label occurs an odd number of times in $\chi(e_1), \chi(e_2), \ldots, \chi(e_k)$, \label{condition:path_2}
        \item for any for any two edges of the same colour $c$, any path between them that has no edges of colour $c$ has an even length. \label{condition:odd_distance_2}
 \end{enumerate}

     Without loss of generality, let edge $v_0v_1$ have label $1$. By condition~\ref{condition:path_2} on the path $v_0v_1v_2$, edge $v_1v_2$ has a different label to $1$, say $2$. 
     
     Edge $v_2v_3$ cannot have label $2$ by condition~\ref{condition:path_2} on the path $v_1v_2v_3$, and cannot have label $1$ by condition~\ref{condition:odd_distance_2} with edge $v_0v_1$. So it must have a new label, say $3$.

     Similarly, edge $v_2v_8$ cannot have label $2$ or $3$ by condition~\ref{condition:path_2} on the paths $v_1v_2v_3$ and $v_1v_2v_8$ respectively, and cannot have label $1$ by condition~\ref{condition:odd_distance_2} with edge $v_0v_1$. So it must have a new label, say $4$.

     Edge $v_3v_4$ cannot have label $3$ by condition~\ref{condition:path_2} on the path $v_2v_3v_4$, and cannot have label $2$ by condition~\ref{condition:odd_distance_2} with edge $v_1v_2$. If it had a label other than $1,2,3$ then the edges $v_0v_{1}, v_{1}v_{2}, v_{2}v_{3}, v_{3}v_{4}$ would all have distinct labels and we would be done. So we may suppose that $v_2v_3$ has label $1$.

    Edge $v_4v_5$ cannot have label $1$ by condition~\ref{condition:path_2} on the path $v_3v_4v_5$, and cannot have label $3$ by condition~\ref{condition:odd_distance_2} with edge $v_2v_3$. If it had a label other than $1,3,4$ then edges $v_{8}v_{2}, v_{2}v_{3}, v_{3}v_{4},v_4v_5$ would all have distinct labels and we would be done. So we may suppose that $v_2v_3$ has label $4$. 

    Edge $v_5v_6$ cannot have label $4$ by condition~\ref{condition:path_2} on the path $v_4v_5v_6$. It cannot have label $1$ or $2$ by condition~\ref{condition:odd_distance_2} with edges $v_3v_4$ and $v_1v_2$ respectively.  Therefore it must have a label distinct from $1,2,4$.  If it had a label other than $1,2,3,4$ then the edges of the path from $v_0$ to $v_6$ would have four labels used exactly once and we would be done. So we may suppose that $v_2v_3$ has label $3$. See Figure~\ref{fig:induced_base_case} for an illustration of the labelling forced up to this point.

    \begin{figure}[h]
    \centering
    \includegraphics[scale=1.1]{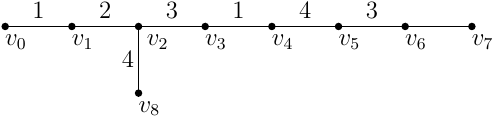}
    \caption{The edge-labels forced on $T$ during the proof of Lemma~\ref{lem:base_case_induced}.}\label{fig:induced_base_case}
    \end{figure}

    Finally, edge $v_6v_7$ cannot have label $3$  by condition~\ref{condition:path_2} on the path $v_5v_6v_7$. It cannot have label $4$ by condition~\ref{condition:odd_distance_2} with edge $v_4v_5$. It also cannot have label $1$ by condition~\ref{condition:path_2} on the path $v_8v_2v_3v_4v_5v_6v_7$. Therefore $v_6v_7$ must have a label other than $1,3,4$.
    
    Then the edges $v_{3}v_{4},v_4v_5, v_5v_6, v_6v_7$ all have distinct labels and we are done.
 \end{claimproof}

 Let $f$ be an embedding of $T$ into a layer $L$ of a hypercube. Let  $\chi: E(G) \rightarrow \mathbb{N}$ be the edge-labelling where $\chi(uv)$ is the direction of the edge $f(u)f(v)$, which is layered by Lemma~\ref{lem:labelling_for_layer}.  By Claim~\ref{claim:label_T}, there is some pair of points $\{x,y\} \in P$ where there are four labels used exactly once on the path between them. Therefore, for this $x,y$ we have $d(f(x),f(y)) = 4$, concluding the proof of Lemma~\ref{lem:base_case_induced}.
\end{proof}

We will use Lemma~\ref{lem:base_case_induced} together with Lemma~\ref{lem:recursion} to prove Theorem~\ref{thm:induced}, which we restate below for convenience.

\induced*
\addtocounter{theorem}{1}
\begin{proof}
    
   Fix  $k > 4$ to be even, and let $r = (k/2)-1 > 1$.

   We define a sequence $(T_1,P_1), (T_2,P_2),\ldots, (T_{r},P_{r})$ as follows: 
   Let $(T_1,P_1)$ be the tree and set of pairs of vertices given in Lemma~\ref{lem:base_case_induced}. 
    Given $(T_{i-1},P_{i-1}$ is $2(i-1)$-distance separating, let $(T_i,P_i)$ be as given by Lemma~\ref{lem:recursion}, so that $(T_i,P_i)$ is $2i$-distance-separating.
    Specifically, we let $T_i$ be the tree formed from $T_{i-1}$ by adding $s$ new vertices adjacent to 
    $x$, for each $x \in T_{i-1}$. For $\{x,y\} \in P_{i-1}$, let $P'_{xy}$ be the set of all pairs consisting of one of these new vertices adjacent to $x$ and one of the new vertices adjacent to $y$. Finally, let $P_i$ be the union of $P'_{xy}$ taken over all pairs $\{x,y\}$ in $P_{i-1}$.
    
    Note that by construction, for all $i>1$ all vertices in a pair in $P_i$ are leaves of the tree $T_i$. We claim that for all $1 \le i \le \ell$, there  exists some $n$ and some embedding $f^{(i)}$ of $T_i$ in $Q_n$ such that: 
    \begin{enumerate}[(i)]
            \item $d(f^{(i)}(x),f^{(i)}(y)) = 2i$ for all pairs $\{x,y\} \in P_{i}$, \label{eq:distance}
            \item for any pair of points $\{x,y\}$ in $P_i$, each direction is used on at most two edges on the path from $f^{(i)}(x)$ to $f^{(i)}(y)$ in $f^{(i)}(T_i)$, and \label{eq:direction_used_twice}   
            \item $f^{(i)}(T_i)$ is an induced subgraph of $Q_{n}$. \label{eq:induced}
    \end{enumerate}
    All of these hold for $T_1$ where $f^{(1)}$ is the embedding given in Lemma~\ref{lem:base_case_induced}. Consider the case $i>1$ and take $f^{(i)}$ as given in the proof of Lemma~\ref{lem:recursion}, so that all edges in $E(T_i)\setminus E(T_{i-1})$ get new distinct directions under $f^{(i)}$ and the remaining edges of inherit the directions given by $f^{(i-1)}$.
    
    Property \ref{eq:distance} follows from the proof of Claim~\ref{claim:recursion_cubical}.  
    
    For  $\{x,y\}$ in $P_i$, there is some $\{u,v\} \in P_{i-1}$ such that $x$ is attached to $u$ and $y$ is attached to $v$. Thus property~\ref{eq:direction_used_twice} then holds by induction, as the new edges $f^{(i)}(x)f^{(i)}(u)$ and $f^{(i)}(y)f^{(i)}(v)$ receive new directions, and the other edges of the path inherit the directions given by $f^{(i-1)}$.
    
    Finally, property \ref{eq:induced} also follows by induction, using the fact that all new edges are given a unique new direction. 
   
   Let $s = k+1$.
    We define a graph $G$ from $T_{r}$ as in the proof of Theorem~\ref{thm:main}, where for each pair of vertices $\{x,y\}$ in $P_{r}$ we add $s$ paths of length $k$ between $x$ and $y$ that are vertex disjoint other than at the endpoints. Let $S_{xy}$ be this set of paths.

    As $T_{r}$ is a tree, any cycle in $G$ must use the entirety of one of these added paths and so clearly $G$ has girth greater than $k$.

    The proof that $G$ is not layered is identical to the proof of Claim~\ref{claim:main_layered}. Thus all that remains is to prove the following claim.

    \begin{claim}
        $G$ is an induced subgraph of some hypercube.
    \end{claim}
    \begin{claimproof}
        
        We define a cubical edge-labelling of $G$ in order to apply Lemma~\ref{lem:labelling_for_cube}. Let $f^{(\ell)}$ be the embedding of $T_{r}$ into $Q_n$. By Lemma~\ref{lem:labelling_for_cube}, this embedding corresponds to a cubical edge-labelling of $T_{r}$. The edges of $G$ contained in the underlying copy of $T_{r}$ (that is, those edges that are not in an added path) inherit these labels. 
            
        Fix $\{x,y\} \in P_r$. Let $I_{xy}$ be the set of coordinates that $f^{(\ell)}(x)$ and $f^{(\ell)}(y)$ differ on. Note that $|I_{xy}| = d(f^{(\ell)}(x),f^{(\ell)}(y)) = k-2$.           
          
        We will label the central $k-2$ edges of each path in $S_{xy}$ by the elements of $I_{xy}$ as in the proof of Claim~\ref{claim:main_cubical}, but this time the order will matter. By property \ref{eq:direction_used_twice}, we know that each label is used at most twice on the path from $x$ to $y$ in $T$. Thus each label in $I_{xy}$ is used exactly once, as it is used an odd number of times. Let $\ell_1, \ell_2, \ldots, \ell_{k-2}$ be the labels in $I_{xy}$ in the order they are used on the path from $x$ to $y$. 

        For each path $Q$ in $S_{xy}$, we introduce a new edge label, $\ell(Q)$, that is distinct from those used elsewhere in $G$. 
        Then going from $x$ to $y$, we label the edges of $Q$ by $\ell(Q), \ell_2, \ell_3, \ldots, \ell_{k-2}, \ell_1, \ell(Q)$ in that order.
      
        Repeat this for each pair $\{x,y\} \in P_r$, so that every edge of $G$ has a label. The verification that the resulting edge-labelling is cubical is identical to the proof of Claim~\ref{claim:main_cubical}. We therefore only need to justify that $f(G)$ is an induced subgraph of a hypercube.

        Suppose for a contradiction that there exist vertices $u,v \in G$ where $f(u)$ and $f(v)$ are adjacent in the hypercube but $u$ and $v$ are not adjacent in $G$. This means that on any path from $u$ to $v$, there is exactly one label that is used an odd number of times. 

        Since $f^{(\ell)}$ is an embedding of $T_{r}$ into an induced subgraph of a hypercube, $u$ and $v$ cannot both lie in $T_{r}$. We may assume that $u$ does not lie in $T_{r}$ and so $u$ is in the interior of a path $Q$ in $S_{xy}$ for some $\{x,y\} \in P_r$.

        If $v = x$, then the subpath of $Q$ from $u$ to $v$ uses the labels $\ell_2$ and $\ell(Q)$ exactly once. Similarly, if $v = y$, then the subpath of $Q$ from $u$ to $v$ uses the labels $\ell_1$ and $\ell(Q)$ exactly once. 
        
        If $v$ is in the interior of the same path $Q$ then there are at least two labels used exactly once on the subpath between $u$ and $v$, as $\ell_1, \ell_2, \ldots, \ell_{k-2}$ are distinct.         
        If $v$ is in the interior of some other added path $Q'$, then the path from $u$ to $v$ uses at least two labels used exactly once: namely, the new labels $\ell(Q)$ and $\ell(Q')$.
    
        So we may assume that $v$ is a vertex that lies in $T_{r}\setminus\{x,y\}$. As $x$ is a leaf of $T_{r}$ and its adjacent edge has a label that is unique, this label must be $\ell_1$. In particular, the path in $T_{r}$ from $x$ to $v$ uses label $\ell_1$. If the subpath of $Q$ from $u$ to $x$ does not use label $\ell_1$, then we obtain a path from $u$ to $v$ via $x$ there are at least two labels used exactly once: namely, $\ell_1$ and $\ell(Q)$. 

        The only remaining case is if the subpath of $Q$ from $u$ to $x$ uses label $\ell_1$, which implies that $u$ is adjacent to $y$ and in particular, the path from $u$ to $y$ does not use label $\ell_{k-2}$. As $y$ is a leaf of $T_{r}$ and its adjacent edge has a label that is unique, this label must be $\ell_{k-2}$. In particular, the path in $T_{r}$ from $x$ to $v$ uses label $\ell_{k-2}$. Therefore we obtain a path from $u$ to $v$ via $y$ where there are at least two labels used exactly once: namely, $\ell_{k-2}$ and $\ell(Q)$. 
    \end{claimproof}
This concludes the proof of Theorem~\ref{thm:induced}.
\end{proof}

It would be interesting to know what dimension of cube is needed to
obtain a non-layered subgraph of girth $k$. The construction above
gives a dimension of the form $k^k$, but perhaps there are even
examples when $n$ is only polynomial in $k$. It would also be nice
to know if the number of vertices in such a cubical graph may be
taken to be just exponential, or even polynomial, in $k$.

\printbibliography

\Addresses
\end{document}